\documentclass[11pt]{article}
\usepackage{amsmath,amsthm,latexsym,amssymb}
\usepackage{amssymb, amsmath}

\newtheorem{thm}{Theorem}[section]
\newtheorem{prop}[thm]{Proposition}
\newtheorem{lem}[thm]{Lemma}

\newtheorem{rmk}{Remark}[section]
\newtheorem{dfn}{Definition}[section]

\newcommand{\R}{{\Bbb R}}

\newcommand{\N}{{\Bbb N}}
\newcommand{\sfe}{{S^{n-1}}}
\newcommand{\sfedue}{{S^2}}

\newcommand{\Hsfe}{{\mathcal{H}}^{n-1}}
\newcommand{\Hdue}{{\mathcal H}^2}
\newcommand{\coco}{{\mathcal K}^n}

\newcommand{\mw}{\mathsf{b}}
\newcommand{\wn}{\omega_n}
\newcommand{\s}{\textsf{s}}

\def\Sn{\mathcal{S}_n}

\begin{document}

\begin{title}{Convexity of solutions and Brunn-Minkowski inequalities for Hessian equations in $\R^3$}
\end{title}

\date{}

\author{Paolo Salani
\footnote{Dipartimento di Matematica ``U. Dini'',
viale Morgagni 67/A, 50134 Firenze, Italy;
salani@math.unifi.it}
}

\maketitle

\begin{abstract}
\noindent By using Minkowski addition of convex functions, we prove convexity and rearrangement
properties of solutions to some Hessian equations in $\R^3$  and Brunn-Minkowski and isoperimetric inequalities for
related functionals.
\end{abstract}



\section{Introduction}
Convexity properties of solutions to partial differential equations are an interesting issue of investigations
since many years and to compile an exhaustive bibliography is almost impossible.
A good reference book has been for a long time the monograph by Kawohl \cite{K}, but a quarter of a century has now passed since
its publication and new techniques and results appeared along these years.
I just recall the ones that are mostly connected to the present paper and address the interested readers
to \cite{SaLectNotes} for more references. In particular, in 1985 Caffarelli-Friedman \cite{CaffF}
(and Singer-Wong-Yau-Yau \cite{SWYY})
devised a {\em microscopic} technique, based on a smart combination of a suitable {\em constant rank theorem}
and continuity method, which opened the way to many results and improvements, see for instance
\cite{BianGuan, BianGuan2, CGM, GuanMa, Kore, Korelewis, LiuMaXu,  MaXu} and references therein.
In 1997 Alvarez-Lasry-Lions \cite{ALL} developed instead a {\em macroscopic} technique based on the convex envelope,
which found its level sets counterpart in \cite{CS1} and \cite{BLS}.
Hence convexity properties of solutions to many elliptic problems are now well understood.
However, some difficult interesting problems are still unsolved to our knowledge, especially for fully nonlinear operators.

Here we will deal in particular with the following Dirichlet problems for Hessian equations:
\begin{equation}\label{eqL}
\left\{\begin{array}{ll}
S_k(D^2u)=\Lambda_k(\Omega)(-u)^k\quad&\text{in }\Omega\,,\\
u=0\quad&\text{on }\partial\Omega\,,\\
u<0\quad&\text{in }\Omega\,,
\end{array}
\right.
\end{equation}
and
\begin{equation}\label{eq1}
\left\{\begin{array}{ll}
S_k(D^2u)=1\quad&\text{in }\Omega\,,\\
u=0\quad&\text{on }\partial\Omega\,,
\end{array}
\right.
\end{equation}
where $\Omega$ is a bounded convex domain of $\R^n$ and $S_k(D^2 u)$ is the $k$-th elementary symmetric function of the eigenvalues of $D^2 u$,
$k\in\{1,\dots,n\}$.

Notice that, when $k=1$ we get back to Poisson equation, while $k=n$ carries out the well-known Monge-Amp\`ere equation.
For $k\geq 2$ the $S_k$ operator is fully nonlinear and it is not elliptic unless when restricted to a suitable class of
admissible functions, the so called {\em $k$-convex functions} (see Section 2 for more details).

The existence and uniqueness of a classical solution of \eqref{eq1}, when $\Omega$ is a {\em $(k-1)$-convex domain},
was proved by Caffarelli, Nirenberg and Spruck in the seminal paper \cite{cns}, while Wang \cite{W} considered problem \eqref{eqL}, proving
the existence of a (unique up to a scalar factor) $k$-convex solution $u\in C^\infty(\Omega)\cap C^{1,1}(\overline\Omega)$
for the eigenvalue
$$
\Lambda_k(\Omega)=\inf\left\{\frac{-\int_\Omega w\,S_k(D^2w)\,dx}{\int_\Omega|w|^{k+1}dx}\,:\,w\in\Phi_{k,0}(\Omega),
\,w<0\text{ in }\Omega\right\}\,,
$$
where $\Phi_{k,0}$ is the set of $k$-convex functions in $C^2(\Omega)\cap C(\overline\Omega)$, vanishing on $\partial\Omega$.
\medskip

Notice that for $k=2$ (which is the case we will deal with here) the appropriate class of sets where to solve
the Dirichlet problems \eqref{eqL} and \eqref{eq1}
is the class of $1$-convex sets, i.e. sets with positive mean curvature. For convenience, we will denote
by $\mathcal{C}_1^+$ {\em the class of convex subsets of $\R^n$ with boundary of class $C^{3,1}$ having everywhere
positive mean curvature} (equivalently, convex sets $\Omega$ with $\kappa_{n-1}>0$ at every boundary point, where
$\kappa_1(x)\leq\dots\leq\kappa_{n-1}(x)$ are the principal curvatures of $\partial\Omega$ at $x$).
\medskip

Convexity properties of solutions of \eqref{eqL} and \eqref{eq1}, in the case $k=2$, $n=3$, are treated in \cite{LiuMaXu, MaXu}; in these papers the authors use the constant rank technique.
\medskip

Here I give new proofs of the results of \cite{LiuMaXu} and \cite{MaXu} and prove some new results.
In particular, regarding the convexity of the solutions to problems \eqref{eq1} and
\eqref{eqL}, I give new proofs of the following two theorems.

\begin{thm}\label{ThmConvL}(\cite[Theorem 1.1]{LiuMaXu})
If $\Omega\subset\R^3$ is
a $\mathcal{C}_1^+$ set, $k=2$ and $u\in C^2(\Omega)\cap C(\overline\Omega)$ is an admissible solution of
\eqref{eqL}, then the function $v=-\log(-u)$ is convex.
\end{thm}

\begin{thm}\label{ThmConv1}(\cite[Theorem 2]{MaXu})
If $\Omega\subset\R^3$ is a $\mathcal{C}_1^+$ set, $k=2$ and $u$ is the admissible classical solution of \eqref{eq1},
then the function $v=-\sqrt{-u}$ is convex.
\end{thm}
It is also possible to generalize the previous theorems to the solutions of the problem
\begin{equation}\label{eqp}
\left\{\begin{array}{ll}
S_2(D^2u)=\lambda\,(-u)^p\quad&\text{in }\Omega\,,\\
u=0\quad&\text{on }\partial\Omega\,,\\
u<0\quad&\text{in }\Omega\,.
\end{array}
\right.
\end{equation}
The existence of a solution to \eqref{eqp} for every $p\in(0,2)\cup(2,+\infty)$ and $\lambda>0$ is proved in \cite{CW1}, see also \cite{CW2, W2}
(notice that $k=2>n/2=3/2$).
In Section 8 I prove the following theorem, together with some other connected results.
\begin{thm}\label{ThmConvp}
Let $p\in(0,2)$, $\lambda>0$. If $\Omega$ is a  $\mathcal{C}_1^+$ subset of $\R^3$, then there exists an admissible classical solution $u$ of
\eqref{eqp} such that the function $v=-(-u)^{(2-p)/4}$ is convex.
\end{thm}

In \cite{LiuMaXu} it is also proved a Brunn-Minkowski inequality for $\Lambda_2(\Omega)$ in $\R^3$,
which is in fact the main result of that paper.
Brunn-Minkowski type inequalities for variational functionals and the convexity properties of the solutions to the
related Dirichlet problems are strongly connected, as showed here and in \cite{SaLectNotes}.
Here I prove a Brunn-Minkowski inequality (in the case $n=3$, $k=2$) for the variational functional $\tau_k$
related to problem \eqref{eq1} and defined as follows
\begin{equation}\label{deftauk}
\frac{1}{\tau_k(\Omega)}=\inf\left\{\frac{-\int_\Omega w\,S_k(D^2w)\,dx}{(\int_\Omega|w|\,dx)^{k+1}}\,:\,w\in\Phi_{k,0}(\Omega)\right\}\,.
\end{equation}
Notice that, $S_1(D^2u)=\Delta u$, then $\Lambda_1(\Omega)$ is the first Dirichlet eigenvalue of the Laplacian and
$\tau_1(\Omega)$ is the so called {\em torsional rigidity} of $\Omega$, then we will refer to $\tau_k(\Omega)$ as the
{\em $k$-torsional rigidity} of $\Omega$. I recall that the Brunn-Minkowski inequality for $\Lambda_1$ is a classical result by Brascamp and Lieb \cite{bl}, while the Brunn-Minkowski inequality for $\tau_1$ was proved by Borell in \cite{b2}.
Moreover $S_n(D^2u)=\det(Du)$, then $\Lambda_n$ is the Dirichlet eigenvalue of the Monge-Amp\`ere operator and the
related Brunn-Minkowski inequality was proved in \cite{Salani}; a proof of the Brunn-Minkowski inequality for $\tau_n$
is obtained in \cite{Hart} and it can also follow from the same \cite{Salani}.
More references and more details about Brunn-Minkowski inequalities for variational functionals are presented in Section 2.

Precisely, here I prove the following.
\begin{thm}\label{ThmBM1}
Let $\Omega_0$ and $\Omega_1$ be $\mathcal{C}_1^+$ subsets of $\R^3$ and $t\in[0,1]$.
Then
\begin{equation}\label{BM1}
\tau_2((1-t)\Omega_0+t\Omega_1)^{\frac{1}{10}}\geq (1-t)\tau_2(\Omega_0)^{\frac{1}{10}}+
t\tau_2(\Omega_1)^{\frac{1}{10}}\,.
\end{equation}
Moreover, equality holds in \eqref{BM1} if and only if $\Omega_0$ and $\Omega_1$ are homothetic.
\end{thm}
Equivalently, the previous theorem asserts that the operator $\tau_2^{\frac{1}{10}}$ is concave with respect to Minkowski
addition in the class of $\mathcal{C}_1^+$ sets in $\R^3$.
\medskip

In Section 8 it is also proved a generalization of the above theorem to functionals related to problem \eqref{eqp},
namely Theorem \ref{ThmBMp}.
\medskip

Notice that, following a standard procedure implemented in \cite{BiaSal} for the Bernoulli constant, the Brunn-Minkowski
inequalities for $\Lambda_2$ and $\tau_2$ leads to the following Urysohn's inequalities
\begin{equation}\label{UryL}
\Lambda_2(\Omega)\geq\Lambda_2(\Omega^\sharp)
\end{equation}
and
\begin{equation}\label{Ury1}
\tau_2(\Omega)\leq\tau_2(\Omega^\sharp)\,,
\end{equation}
where $\Omega^\sharp$ is a ball with the same mean-width of $\Omega$; moreover, equality holds in any one
of the above inequalities if and only if $\Omega=\Omega^\sharp$.
On the other hand, thanks to Theorem \ref{ThmConvL} and Theorem \ref{ThmConv1}, implying the convexity of the level sets of
the solutions of \eqref{eqL} and \eqref{eq1} for $k=2$ and $n=3$, it is now possible to carry on a
symmetrization by quermassintegrals (precisely by surface measure) argument from \cite{Trudi, Tso} and to give a
complete proof
of the following inequalities:
\begin{equation}\label{IsoL}
\Lambda_2(\Omega)\geq\Lambda_2(\Omega^*)
\end{equation}
and
\begin{equation}\label{Iso1}
\tau_2(\Omega)\leq\tau_2(\Omega^*)\,,
\end{equation}
where $\Omega^*$ is a ball with the same surface area as $\Omega$.

Notice that \eqref{IsoL} and \eqref{Iso1} are better than (in the sense that they imply) \eqref{UryL} and \eqref{Ury1},
respectively,
due to classical isoperimetric inequalities for quermassintegrals which imply that $\Omega^*\subset\Omega^\sharp$.
See Section 8 for details.
\bigskip

The basic technique here adopted to obtain all these results is in some sense a refinement of the technique of \cite{ALL} and
it is based on the Minkowski addition of convex functions (that is nothing else than the classical infimal convolution operation),
which provides a sort of continuous rearrangement and permits then to obtain results that are typical of rearrangement's techniques,
like the following.
\begin{thm}\label{ThmUrysohn}
Let $\Omega\subset\R^3$ be a $\mathcal{C}_1^+$ set and let $\Omega^\sharp$ be a ball with the same mean-width of $\Omega$.
Denote by $u$ the solution of \eqref{eq1} in $\Omega$ and by $u^\sharp$
the solution in $\Omega^\sharp$. Then
\begin{equation}\label{ineqlpnorm}
\|u\|_{L^p(\Omega)}\leq \|u^\sharp\|_{L^p(\Omega^\sharp)}\quad\text{for every }p\in(0,+\infty]\,.
\end{equation}
Moreover, equality holds for any $p\in(0,+\infty)$ if and only if $\Omega$ is a ball.
\end{thm}
This result is completely
new, to my knowledge; however it should be compared with the results of \cite{Trudi} and \cite{Tso}, see Section 8.
\bigskip

In any case, in my opinion, the most important novelty of this paper is probably not in the new results here contained, but
in that it presents a technique which allows to unify the proof of convexity properties of solutions of Dirichlet problems
(Theorems \ref{ThmConvL} and \ref{ThmConv1}) and the proof of Brunn-Minkowski type inequalities (Theorems \ref{ThmBM1} and \cite[Theorem 2]{LiuMaXu}),
as well as results like Theorem \ref{ThmUrysohn}.
Notice also that Theorem \ref{ThmBM1} and Theorem \ref{ThmUrysohn} are in fact corollaries of a more general result,
namely Theorem \ref{ThmBMRearr},
which can be considered for this reason the main result of this paper and has also other interesting consequences,
as we shall see.
The range of application of this technique is very large
and I refer to \cite{SaLectNotes} for more details and general results.
\bigskip

The paper is organized as follows. In Section 2, I introduce notation and recall some basic facts about
Minkowski addition of convex functions, Hessian operators and Brunn-Minkowski inequalities.
In Section 3 I prove Theorem \ref{ThmConvL}. Section 4 contains the proof of Theorem \ref{ThmConv1}. Section 5 is devoted to
Theorem \ref{ThmBMRearr}.
In Section 6 I prove Theorem \ref{ThmBM1}. Section 7 contains
the proof of Theorem \ref{ThmUrysohn}. Finally, In Section 8 I prove Theorem \ref{ThmConvp}
and a Brunn-Minkowski inequality for the
functional $\tau_{2,p}$ which generalize $\tau_2$, as well as an analogue of Theorem \ref{ThmBMRearr} for problem \eqref{eqp}
and the isoperimetric inequalities \eqref{UryL}-\eqref{Iso1}; I also
give some comments and remarks and suggest some open problems.

\section{Preliminaries}

\subsection{Notation}

Throughout the paper, $\Omega$ (possibly with subscripts) denotes a bounded domain in
${\bf R}^3$; in general $\Omega$ will be also convex.
We say that $\Omega$ is of class $C^2_+$ if its boundary $\partial\Omega$ is of
class $C^2$ and the Gauss curvature at every point of
$\partial\Omega$ is strictly positive. We say that $\Omega$ is of class $\mathcal{C}_1^+$ if it is convex, with boundary
of class $C^{3,1}$ having everywhere positive mean curvature.

Let $u:\Omega\rightarrow {\bf R}$ be a twice differentiable
function; for $i,j=1,2,3$ we set $u_i=\frac{\partial u}{\partial
x_i}$ and $u_{ij}=\frac{\partial^2 u}{\partial x_i \partial
x_j}$ and we denote by $Du$ the gradient of $u$ and
by $D^2u$ its Hessian matrix. We say that $u$ is of class
$C^{2,+}(\Omega)$ if $u\in C^2(\Omega)$ and $D^2 u(x)>0$
for every $x\in\Omega$.

We denote by $\Sn$ the space of the real symmetric $n\times n$ matrix.
If $A\in\Sn$ we write $A\geq 0$ if $A$ is positive semidefinite and $A>0$ if $A$ is positive definite.
Then we set $\Sn^+=\{A\in\Sn\,:\,A\geq 0\}$ and $\Sn^{++}=\{A\in\Sn\,:\,A>0\}$.
By $A\geq B$, we mean $A-B\geq 0$. If $A\in\Sn^{++}$ we denote by $A^{-1}$ its inverse matrix.

\subsection{Hessian operators and Hessian equations}
Let $A=(a_{ij})\in\Sn$ and denote by
$\lambda_1,...,\lambda_n$ its eigenvalues. For $k\in
\left\{1,...,n\right\}$, the $k-$th elementary symmetric function
of $A$ is
$$S_k(A)=S_k(\lambda_1,...,\lambda_n)=\sum_{1\leq i_1
<\cdots<i_k\leq n } \lambda_{i_1}\cdots \lambda_{i_k}.$$
Note that $S_k(A)$ is just the sum of all $k \times k$ principal minors of $A$.
In particular $S_1(A)=\text{tr}(A)$ is the trace of $A$ and $S_n(A)=\text{det}(A)$ is its determinant.

The
operator $S_k^{1/k}$, for $k=1,...,n$, is homogeneous of degree
$1$ and it is increasing and concave if restricted to
$$\Gamma_k=\{A\in\Sn\,:\,S_i(A)> 0\text{ for }i=1,\dots,k\}\,.$$
Moreover $S_k(A^{-1})^{-1/k}$ is concave in the class $\Sn^{++}$.

The following algebraic lemma from \cite{LiuMaXu} is crucial to this paper.
I recall it here for the convenience of the reader.
\begin{lem}\label{lemmaLMX}(\cite[Proposition 4.3]{LiuMaXu})
Let $P\in\mathcal{S}_3^+$ be a fixed matrix, $P\neq 0$. Then the functions
$f(A)=\frac{S_1(PA^{-1})}{S_2(A^{-1})}$ and $g(A)=S_1(PA^{-1})^{-1}$ are concave in $\mathcal{S}_3^{++}$.
\end{lem}
The statement of \cite[Proposition 4.3]{LiuMaXu} may look slightly weaker than the above statement, but the claim
of Lemma \ref{lemmaLMX} is precisely what the authors prove in \cite{LiuMaXu}.

A Hessian equation is an equation of the following type
\begin{equation}\label{hessianeq}
S_k(D^2u)=f(x,u,Du)\quad\text{ in }\Omega\,,
\end{equation}
with $k\in\{1,\dots,n\}$.
Hessian equations have been the subject of many investigations after the seminal paper by Caffarelli, Nirenberg and Spruck
\cite{cns}; see \cite{W2} for more details and references. Here I just recall that equation \eqref{hessianeq}, when $k>1$, is elliptic only when restricted to the following class
of functions
$$
\Phi_k(\Omega)=\{u\in C^2(\Omega)\,:\,D^2u(x)\in\Gamma_k\text{ for every }x\in\Omega\}\,.
$$
Functions in $\Phi_k$ are called {\em $k$-convex} or {\em admissible} for $S_k$. For instance, when $k=n$
(that is for Monge-Amp\`{e}re equation) only
convex functions are admissible for $S_n$.

In connection with $k$-convex functions, it is useful to recall also the notion of $k$-convex set: a set $\Omega\subset\R^n$ of class
$C^2$ is said $k$-convex if $(\kappa_1(x),\dots,\kappa_{n-1})\in\Gamma_k$ for every $x\in\partial\Omega$, where
$\kappa_1(x),\dots,\kappa_{n-1}(x)$ are the principal curvature of $\partial\Omega$ at $x$.
Clearly, $(n-1)$-convex sets are $C^2_+$ sets.
A (regular) level set of a $k$-convex function is $(k-1)$-convex, see \cite{cns}.

\subsection{Minkowski addition of convex functions and the convex envelope of a non-convex function}
Let
$\Omega_0$ and $\Omega_1$ be two convex sets in $\R^n$,
$u_0$ and $u_1$ two convex functions in $\Omega_0$ and $\Omega_1$ respectively.
For $t\in[0,1]$, the {\em inf-convolution} $\tilde{u}_t$ of $u_0$ and $u_1$ is defined
in the convex set $\Omega_t=(1-t)\Omega_0+t\,\Omega_1$ as follows
\begin{equation*}
\tilde{u}_t(x)=\inf \left \{(1-t)u_0(x_0)+tu_1(x_1):\;x_i \in
\Omega_i,\,i=0,1,\,x=(1-t)x_0+tx_1\right\}\,.
\end{equation*}

As $u_0$ and $u_1$ are convex, the function $\tilde{u}_t$  is convex, see
 \cite[Section 5]{Rocka}, and its epigraph coincides with the Minkowski linear combination of the epigraphs of $u_0$ and $u_1$ (see \cite{Salani}), i.e.
\begin{equation}\label{MinkAddGraph}
\begin{array}{ll}
\{(x,s)\in\R^{n+1}\,:\,x\in\Omega_t,\,s\geq\tilde{u}_t(x)\}\\
\quad=(1-t)\{(x,s)\,:\,x\in\Omega_0,\,s\geq u_0(x)\}+
t\,\{(x,s)\,:\,x\in\Omega_1,\,s\geq u_1(x)\}\,.
\end{array}
\end{equation}
For this reason I will refer to this operation with the expression {\em Minkowski combination} of $u_0$ and $u_1$ instead of the more usual
{\em infimal convolution.}

Of course, one can consider the combination of more than two functions: let $N\in\N$ and $t=(t_0,\dots,t_N)\in\Upsilon_N$, where
$$\Upsilon_N=\left\{(s_0,\dots,s_N)\,:\,s_i\geq 0\,\text{ for }i=0,\dots,N,\,\sum_{i=0}^Ns_i=1\right\};
$$
for $i=0,\dots,N$, let $\Omega_i$ be a convex subset of $\R^n$ and $u_i$ a convex function defined in $\Omega_i$; then we
set $\Omega_t=\sum_{i=0}^Nt_i\Omega_i$ and we define the function $\tilde u_t$ as follows:
\begin{equation}\label{defut}
\tilde{u}_t(x)=\inf \left \{\sum_{i=0}^Nt_iu_i(x_i):\;x_i \in
\Omega_i\text{ for }i=0,\dots,N,\,x=\sum_{0=1}^Nt_ix_i\right\}\,.
\end{equation}

The following lemma is a slight improvement of \cite [Lemma 2.1]{CoCuSa}.
\begin{lem} Let $N\in\N$. For $0=1,\dots,N$, let $\Omega_i\subset\R^n$ be an open, bounded, convex set
and $u_i\in C^1(\Omega_i)$ be a strictly convex function such
that
\begin{equation}\label{Duinfty}
\lim_{x\rightarrow\partial\Omega_i}|Du_i(x)|=\infty\,.
\end{equation}
Then,
for $t\in\Upsilon_{N}$ (with the notation introduced above), ${\tilde
  u}_t\in C^1(\Omega_t)$ and it is strictly concave; moreover,
for every $x\in\Omega_t$, there exists a
unique $(N+1)$-tuple of points $(x_0,\dots,x_N)\in\prod_{0=1}^N\Omega_i$ such that
\begin{eqnarray}
&x=\sum t_ix_i\,,\label{unicixy}&\\
&\tilde{u}_t(x)=\sum t_iu_i(x_i)\,,
\label{massimante}&\\
&D\tilde{u}_t(x) = Du_0(x_0)=\dots=Du_N(x_N)\,.&
\label{gradienticonvolu}
\end{eqnarray}
If in addition $x\in\Omega_t$ is such that $u_0,\dots,u_N$  are twice
differentiable at the corresponding point $x_0,\dots,x_N$ respectively and
$D^2 u_0(x_0)>0,\dots,D^2 u_N(x_N)>0$, then $\tilde u_t$ is twice
differentiable at $x$ and
\begin{equation}
D^2\tilde{u}_t(x)=\left[\sum_{i=0}^N t_i\left(D^2u_i(x_i)\right)^{-1}\right]^{-1}\,.
\label{hessianeconvolu}
\end{equation}
\label{lemmaconvolu}
\end{lem}
\begin{proof}
This lemma almost coincides with the obvious extension of Lemma 2.1 of \cite{CoCuSa} to the case of $N+1$ functions (apart from considering convex functions in place of concave functions). The only difference is that here we only assume \eqref{Duinfty} in place of $\lim_{x\rightarrow\partial\Omega_i}u_i(x)=+\infty$; then we have just to notice that
\eqref{Duinfty} implies \cite[(11)]{CoCuSa}, that is
$$
D{\tilde
  u}_t(\Omega_t)=Du_0(\Omega_0)=\dots=Du_N(\Omega_N)=\R^n\,.
$$
\end{proof}

Next we notice that the definition of $\tilde u_t$ does not need the functions $u_0,\dots,u_N$ to be convex, as Minkowski addition is well defined and interesting also for non-convex sets. In this case \eqref{MinkAddGraph} still holds, but of course $\tilde u_t$ is in general not convex.

Moreover, when the function $u$ is not convex, it is interesting to consider the case when
$u_0=\dots=u_N=u$ and $\Omega_0=\dots=\Omega_N=\Omega$.
In such a case \eqref{defut} reads
$$
\tilde u_t=\inf\left \{\sum_{i=0}^N t_iu(x_i):\;x_i\in
\Omega,\,i=0,\dots,N,\,x=\sum_{i=0}^N t_ix_i\right\}\,.
$$
The epigraph of $\tilde u_t$ is then the Minkowski linear combination of $N+1$ copies of the epigraph of $u$.
When $N=n$ (or greater), $t$ is not fixed and the infimum in \eqref{defut} is taken also with respect to $t\in\Upsilon_{n}$, then we obtain the function
\begin{equation}\label{deftildeu}
\tilde{u}(x)=\inf \left \{\sum_{i=0}^n t_iu_i(x_i):\;x_0,\dots,x_N \in
\Omega,\,t\in\Upsilon_n,\,x=\sum_{i=0}^N t_ix_i\right\}\,,
\end{equation}
that is the {\em convex envelope} of the function $u$.

Obviously $\tilde u\leq u$ in $\Omega$; the points where $\tilde u(x)=u(x)$ are called {\em contact points} and the set
$$
C_u(\Omega)=\{x\in\Omega\,:\,\tilde u(x)=u(x)\}
$$
is {\em the contact set of $u$ in $\Omega$}.

Notice that $$\inf_\Omega\tilde u=\inf_\Omega u$$ and the minimum points of $u$ (if any) always belong to the contact set, i.e. if $m=\inf_{\Omega}u$ in $\Omega$ and $u(x)=m$, then $x\in C_u(\Omega)$.

\subsection{Brunn-Minkowski inequalities}\label{BMsec}

The Brunn-Minkowski inequality in its
classical formulation regards the volume of convex bodies.
Let $K_0$ and $K_1$ be compact convex
sets in ${\bf R}^n$ with non-empty interior, i.e. {\em convex bodies}, and fix $t\in[0,1]$; then
consider the convex linear combination of these sets:
\begin{equation*}
K_t=(1-t)K_0+t K_1=\{(1-t)x+ty\,|\, x\in K_0\,,\, y\in K_1\}\,,
\end{equation*}
which is still a convex body.
The following inequality holds:
\begin{equation}
\label{intro.1}
V(K_t)^{1/n}\ge(1-t)V(K_0)^{1/n}+tV(K_1)^{1/n}\,,
\end{equation}
where $V$ denotes the $n$-dimensional volume (i.e. the Lebesgue
measure). Moreover, equality holds in (\ref{intro.1}) if and
only if $K_0$ and $K_1$ are {\em homothetic}, i.e. they are equal up to
translations and dilatations.

The validity of the Brunn-Minkowski inequality (\ref{intro.1}) goes in fact far beyond the family of convex bodies, namely it can be extended to the class of measurable sets. Of course it
has a fundamental role in the theory of convex bodies, but its importance extends to many fields of analysis and it is
strongly connected to many other inequalities like the isoperimetric
inequality and the Sobolev inequality (see for instance \cite[Chapter 6]{sc} and the
beautiful survey paper \cite{gardner}).

Let $\coco$ denote the class of convex bodies in ${\bf R}^n$; $\coco$ is
endowed with a scalar multiplication for positive numbers
\begin{equation*}
s\, K=\{s\,x\,|\, x\in K\}\,,\quad K\in\coco\,,\, s>0\,,
\end{equation*}
and with the Minkowski addition
\begin{equation*}
K_0+K_1=\{x+y\,|\, x\in K_0\,,\, y\in K_1\}\,,\quad
K_0,K_1\in\coco\,.
\end{equation*}
The Brunn-Minkowski inequality is then equivalent to the concavity in $\coco$ of $V^{1/n}(\cdot)$, the
$n$-dimensional volume
raised to the power $1/n$; note that $V(\cdot)$ is positively
homogeneous and its order of homogeneity is precisely $n$:
\begin{equation*}
V(s\,K)=s^n\, V(K)\,,\quad \forall s>0\,.
\end{equation*}
These considerations suggest the following:

\begin{dfn} Let $F\,:\,\coco\rightarrow{\bf R}_+$ be a functional,
  invariant under rigid motions of ${\bf R}^n$ and positively
  homogeneous of some order $\alpha\ne 0$. We say that $F$ satisfies a Brunn-Minkowski type
  inequality if
  \begin{equation}\label{BMFdef}
  F^{1/\alpha}\,\text { is concave
  in }\coco\,.
  \end{equation}
\label{def.intro.1}
\end{dfn}

In convex geometry there are many examples of functionals satisfying
a Brunn-Minkowski inequality: the $(n-1)$-dimensional measure of the
boundary, the other {\em quermassintegrals},
etc. (see \cite{gardner} and \cite{sc}, for instance). On the other hand, inequalities
of Brunn-Minkowski type have been proved for functionals coming from
a quite different area: the Calculus of Variations.
The first example in this sense is due to Brascamp and Lieb
who proved that the first eigenvalue of the Laplace operator satisfies a
Brunn-Minkowski inequality (cfr \cite{bl}). Subsequently, Borell proved the
same result for the Newton capacity, the logarithmic capacity (in
dimension $n=2$) and the
torsional rigidity (cfr \cite{Borell1}, \cite{b2} and \cite {b3}
respectively). In \cite{cjl} Caffarelli, Jerison and
Lieb established equality conditions for the Newton capacity.
These results have been recently generalized, improved and developed in various
directions, starting from \cite{colesala}, where it is proved that the $p$-capacity, $p\in(1,n)$,
satisfies a Brunn-Minkowski inequality (including equality conditions).
Other related results are contained  for instance in \cite{Colesanti, CoCuSa, GardHart, Hart, LiuMaXu, Salani, WangXia}.
Obviously, the most relevant result for the present paper is the one contained in \cite{LiuMaXu}.
Notice that in all known cases, equality conditions are the same as in the classical Brunn-Minkowski inequality
for the volume, i.e. equality holds if and only if the involved sets
are (convex and) homothetic.

Due to the homogeneity of the involved functional and thanks to a standard argument, when $\alpha>0$
the Brunn-Minkowski inequality for $F$ is equivalent to any one of the following weaker concavity properties:

(i) $F^\beta$ is concave for some $\beta\in(0,1/\alpha]$;

(ii) $\log F$ is concave;

(iii) $F^\gamma$ is convex for some $\gamma<0$;

(iv) $F$ is quasi-concave, i.e.
$$F(K_t)\geq\min\{F(K_0),\,F(K_1)\}\quad\text{for every }K_0,\,K_1\in\coco\text{ and every }t\in[0,1]\,.$$

Finally, I recall here the Pr\'ekopa-Leindler inequality, which is a functional equivalent form of the Brunn-Minkowksi inequality.
\begin{prop}{\bf (Pr\'{e}kopa-Leindler Inequality)} Let $f,g,h\in
  L^1({\R}^n)$ be nonnegative functions and $t\in(0,1)$. Assume that
$$
h\left((1-t)x+ty\right)\geq f(x)^{1-t}g(y)^t\,,
$$
for all $x,y\,\in {\R}^n$. Then
$$
\int_{{\R}^n}h(x)dx\geq\left(\int_{{\R}^n}f(x)dx\right)^{1-t}\left(\int_{{\R}^n}g(x)dx\right)^t\,.
$$
In addition, if equality holds then $f$ coincides a.e. with a
log-concave function and there exist $C\in{\R}$, $a>0$ and
$y_0\in{\R}^n$ such that
$$
g(y)=C\,f(a y+y_0)\quad\mbox{for almost every
}\;y\in {\R}^n\,.
$$
\label{Prekopaleindler}
\end{prop}
For a proof of the Pr\'{e}kopa-Leindler inequality and precise references,
see for instance \cite{gardner}. The equality condition is
due to Dubuc, see Theorem 12 in \cite{dubuc}.

\section{Proof of Theorem \ref{ThmConvL}}
\begin{proof}

Let
$$
v=-\log(-u)\,,
$$
where $u$ is an admissible solution of \eqref{eqL} (recall that by \cite{W} a $k$-convex solution $u\in C^\infty(\Omega)\cap C^{1,1}(\overline\Omega)$
of \eqref{eqL} exists and it is unique up to multiplication by a positive scalar factor).

Then
$$
u_i=e^{-v}v_i\quad\text{ and }\quad u_{ij}=e^{-v}(v_{ij}-v_iv_j)\quad\text{ for }i,j=1,2,3\,.
$$
Hence $v$ satisfies the following
\begin{equation}\label{eqvL}
\left\{\begin{array}{ll}
S_2(D^2v)-S_1(P(Dv)D^2v)=\Lambda_2(\Omega)\quad&\text{in }\Omega\\
v(x)\to+\infty\quad&\text{as }x\to\partial\Omega\,,
\end{array}
\right.
\end{equation}
where $P(Dv)$ is a positive semidefinite $3\times3$ matrix with entries
$$P_{ij}=|Dv|^2\delta_{ij}-v_iv_j\,,$$
as noticed in \cite{LiuMaXu}.

Let $\tilde v$ be the convex envelope of $v$:
$$
\tilde v(x)=\inf\{\sum_{i=0}^3 t_iv(x_i)\,:\,x_i\in\Omega,\,t\in\Upsilon_3,\,\sum_{i=0}^3t_ix_i=x\}\,.
$$
Notice that $\tilde v\leq v$ by definition, while
\begin{equation}\label{eqmin}
\min_{\overline\Omega}\tilde v=\min_{\overline\Omega}v\,.
\end{equation}
The statement of Theorem \ref{ThmConvL} claims that $\tilde v$ coincides with $v$.
We will prove this by showing that $\tilde v$ actually satisfies (possibly in the viscosity sense) the following
\begin{equation}\label{eqtildevL}
\left\{\begin{array}{ll}
S_2(D^2\tilde v)-S_1(P(D\tilde v)D^2\tilde v)\leq\Lambda_2(\Omega)\quad&\text{in }\Omega\\
\tilde v(x)\to+\infty\quad&\text{as }x\to\partial\Omega\,.
\end{array}
\right.
\end{equation}
Before proving \eqref{eqtildevL}, let us see how this will lead to the desired conclusion.
First we notice that, by a simple observation, since $\tilde v$ is convex, it is an admissible solution of \eqref{eqtildevL};
then $\tilde u=-e^{-\tilde v}$ is an admissible solution of
\begin{equation}\label{eqtildeuL}
\left\{\begin{array}{ll}
S_2(D^2\tilde u)\leq\Lambda_2(\Omega)(-\tilde u)^2\quad&\text{in }\Omega\\
\tilde u=0\quad&\text{on }\partial\Omega\,,
\end{array}
\right.
\end{equation}
whence (multiplying by $-\tilde u$ and integrating over $\Omega$)
$$
\frac{-\int_\Omega \tilde u\,S_2(D^2\tilde u)\,dx}{\int_\Omega|\tilde u|^{3}dx}\leq\Lambda_2(\Omega)\,,
$$
which implies $\tilde u=\lambda u$ for some $\lambda\neq 0$, by \cite{W}. Owing to \eqref{eqmin},
$\lambda=1$ and finally $u=\tilde u$.

To prove \eqref{eqtildevL}, consider a test function $\phi$ touching $\tilde v$ by below at some point $\bar x\in\Omega$,
i.e. a $C^2$ function such that $\phi(\bar x)=\tilde v(\bar x)$ and $\phi\leq\tilde v$ in a neighborhood of $\bar x$.
We have to prove that
\begin{equation}\label{eqphiL}
S_2(D^2\phi(\bar x))-S_1(P(D\phi(\bar x))D^2\phi(\bar x))\leq\Lambda_2(\Omega)\,.
\end{equation}
If $\bar x$ is a contact point for $v$, that is if $\tilde v(\bar x)=v(\bar x)$, there is nothing to prove, since $\phi$ touches also $v$ by below at $\bar x$ and correspondingly
$\psi=-e^{-\phi}$ touches $u$ by below at $\bar x$, then $\psi(\bar x)=u(\bar x)$, $D\psi(\bar x)=Du(\bar x)$ and $D^2\psi(\bar x)\leq D^2u(\bar x)$,
whence
$$
S_k(D^2\psi(\bar x))\leq S_k(D^2u(\bar x))=\Lambda_2(\Omega)u(\bar x)^2=\Lambda_2(\Omega)\psi(\bar x)^2
$$
which in turn is equivalent to
\eqref{eqphiL}.

Hence, assume $\tilde v(\bar x)<v(\bar x)$. First, we notice that $\bar x$ is not a minimum point of $v$ (or $\tilde v$),
whence $D\tilde v(\bar x)\neq 0$. Furthermore, since $v(x)\to+\infty$ as $x\to\partial\Omega$,
the infimum in the definition of $\tilde v$ is in fact a minimum, that is there exist $t\in\Upsilon_3$ and
$x_0,\dots,x_3\in\Omega$ such that
\begin{equation}\label{convenvelope}
\begin{array}{rcl}
&\bar x=\sum t_ix_i\,,&\\
&\tilde{v}(\bar x)=\sum t_iv(x_i)\,,
&\\
&D\tilde{v}(\bar x) = Dv(x_0)=\dots=Dv(x_3)\,,&
\end{array}
\end{equation}
see Lemma \ref{lemmaconvolu}.

Then, letting $p=D\tilde{v}(\bar x) = Dv(x_0)=\dots=Dv(x_3)$, the hyperplane
$$
z-\tilde v(\bar x)=<p,x-\bar x>
$$
coincides with
$$
z-v(x_i)=<p,x-x_i>\quad i=0,\dots,3\,,
$$
and it is a support hyperplane to the graph of $v$ at the points $x_0,\dots,x_N$ (and also to the graph of $\tilde v$ at $\bar x$),
i.e. $v(x)\geq v(x_i)+<p,x-x_i>$ for every $x\in\Omega$, $i=0,\dots,3$, equality holding if $x=x_i$.
Hence
\begin{equation}\label{D2vgeq0}
D^2v(x_i)\geq 0\quad i=0,\dots,3\,.
\end{equation}
In fact, up to an approximation procedure that we will show in detail later, we may assume
\begin{equation}\label{hpcomodo}
D^2v(x_i)>0\quad i=0,\dots 3\,.
\end{equation}
Then there exists $r>0$ such that $D^2v(x)>0$ in $B_i=B(x_i,r)$ for $i=0,\dots,3$. Let us denote by
$v_i$ the restriction of $v$ to $B_i$. Then $\tilde v$ restricted to $\tilde B=B(\bar x,r)$ coincides with the Minkowski
linear combination (with ratio $t$) of the functions $v_0,\dots,v_3$, it is twice differentiable at $\bar x$
and $D^2\tilde v(\bar x)$ satisfies \eqref{hessianeconvolu}, by Lemma \ref{lemmaconvolu}.
Then
$$
D^2\phi(\bar x)\leq\left[\sum_{i=0}^3 t_i\left(D^2u_i(x_i)\right)^{-1}\right]^{-1}\,,
$$
while $D\phi(\bar x)=p$. Then, by applying Lemma \ref{lemmaLMX}, we obtain
\begin{equation}\label{S2divisoS1}
\frac{S_2(D^2\phi(\bar x))}{S_1(P(D\phi(\bar x))D^2\phi(\bar x))}-1\leq\frac{\Lambda_2(\Omega)}{S_1(P(D\phi(\bar x))D^2\phi(\bar x))}\,,
\end{equation}
whence \eqref{eqphiL}.
\bigskip

The proof is essentially concluded. We have just to show how to get rid of \eqref{hpcomodo},
which we assumed for our convenience. We use a similar argument to \cite{CoCuSa}.
For $\varepsilon>0$ and $i=0,\dots,3$ we set
$$
v_{i,\varepsilon}(x)=v_i(x)+\varepsilon \frac{| x
|^2}{2},\;\;x\in B_i\,.
$$
The function $v_{i,\varepsilon}$ is strictly convex in $B_i$ and, thanks to \eqref{D2vgeq0}, it holds
$D^2v_{i,\varepsilon}(x_i)=D^2v_i(x_i)+\varepsilon^2I>0$, where $I$ is the $3\times 3$ identity matrix.
We consider the sup-convolution
$\tilde{v}_{\varepsilon}$ of $v_{0, \varepsilon},\dots,v_{3,
\varepsilon}$.
Clearly $v_{i,\varepsilon}$ converges in $C^2$ norm to $v_i$
in $B_i$, for $i=0,\dots,3$, while
$\tilde{v}_{\varepsilon}$ converges uniformly (actually $C^1$-uniformly by \cite[Theorem 25.7]{Rocka}) to $\tilde{v}$ in
$\tilde B$ as $\varepsilon\to 0$. The test function $\phi_\varepsilon(x)=\phi(x)+\varepsilon \frac{| x
|^2}{2}$ touches $\tilde{v}_{\varepsilon}$ by below at $\bar x$ and converges $C^2$-uniformly to
$\phi$ as $\varepsilon\to 0$, clearly.
Then we repeat the above argument for $\phi_\varepsilon$ and then we let $\varepsilon\to 0$, obtaining \eqref{eqphiL}.

\end{proof}
\begin{rmk}
Notice that we do not really need the $C^2$ regularity, nor the strict $1$-convexity of the set $\Omega$ in the previous proof.
We just need $\Omega$ to be convex.
On the other hand, our assumptions on $\Omega$ guarantee the existence of classical solutions to problem \eqref{eqL},
see \cite{W}.
\end{rmk}

\section{Proof of Theorem \ref{ThmConv1}}
\begin{proof}
By \cite{cns}, there exists a unique $k$-convex solution $u\in C^\infty(\Omega)\cap C^{1,1}(\overline\Omega)$ of \eqref{eq1}.
Let
$$v=-\sqrt{-u}\,.
$$
Then
$$
u_i=-2vv_i\quad\text{ and }\quad u_{ij}=-2vv_{ij}-2v_iv_j\quad\text{ for }i,j=1,2,3\,.
$$
Hence $v$ satisfies the following
\begin{equation}\label{eqv1}
\left\{\begin{array}{ll}
v^2S_2(D^2v)+vS_1(P(Dv)D^2v)=1/4\quad&\text{in }\Omega\\
v=0\quad&\text{on }\partial\Omega\,,
\end{array}
\right.
\end{equation}
where $P(Dv)$ is the same as before.

Notice that, thanks to Hopf's lemma, it holds $|Du|>0$ on $\partial\Omega$ whence
\begin{equation}\label{Dvinfty}
\lim_{x\to\partial\Omega}|Dv(x)|=+\infty\,,
\end{equation}

Now, let $\tilde v$ be the convex envelope of $v$, consider a test function $\phi$ touching $\tilde v$ by below at some point $\bar x\in\Omega$
and notice that, thanks to \eqref{Dvinfty} by Lemma \ref{lemmaconvolu}, there exist points $x_0,\dots,x_3\in\Omega$
and $t\in\Upsilon_3$ such that
\eqref{convenvelope} hold. Then the proof proceeds almost as the proof of Theorem \ref{ThmConvL}.
There are two main (but not big) differences.

The first one is that \eqref{S2divisoS1} has to be replaced by the following
\begin{equation}\label{S2divisoS11}
\phi(\bar x)^2\frac{S_2(D^2\phi(\bar x)}{S_1(P(D\phi(\bar x))D^2\phi(\bar x))}
+\phi(\bar x)\leq\frac{1}{4\,S_1(P(D\phi(\bar x))D^2\phi(\bar x))}\,.
\end{equation}
The latter means that $\tilde v$ satisfies, in the viscosity sense,
$$
\left\{\begin{array}{ll}
\tilde v^2S_2(D^2\tilde v)+\tilde vS_1(P(D\tilde v)D^2\tilde v)\leq1/4
\quad&\text{in }\Omega\\
\tilde v=0\quad&\text{on }\partial\Omega\,,
\end{array}
\right.
$$
which in turn implies that $\tilde u=-\tilde v^2$ is a supersolution of problem \eqref{eq1}. Then
$\tilde u\geq u$, by the Comparison Principle;
on the other hand $\tilde u\leq u$ by the very definition of $\tilde v$, then $\tilde u=u$ and the proof is finished.

The second difference is that, in proving \eqref{S2divisoS11}, we need the convexity of the operator
$F:\R\times\mathcal{S}_3^{++}\to\R$ defined by
$$
F(t,A)=t^2\frac{S_2(A^{-1})}{S_1(P(p)A^{-1})}\,.
$$
To prove this we have just to combine Lemma \ref{lemmaLMX} with case (iii) ($\alpha=2$) of \cite[Lemma A.1]{IS}.

\end{proof}
\begin{rmk}
As for Theorem \ref{ThmConvL}, we assumed regularity and positive mean curvature of $\Omega$ only
to be sure to have existence of classical solutions to problem \eqref{eq1} according to \cite{cns}.
\end{rmk}
\begin{rmk}
Notice that the proof works unchanged for more general equations like
$$
S_2(D^2u)=f(Du/|Du|)
$$
or
$$
S_2(D^2u)=f(Du/\sqrt{-u})\,.
$$
\end{rmk}

\section{The main theorem}
This section contains the main theorem of this paper, which is the following.
\begin{thm}\label{ThmBMRearr}
Let $\Omega_0$ and $\Omega_1$ be two $C^2_+$ subsets of $\R^3$, $t\in[0,1]$ and
$$\Omega_t=(1-t)\Omega_0+t\Omega_1\,.
$$
Let us denote by $u_i$ the solution of (\ref{eq1}) when $\Omega=\Omega_i$, for
$i=0,1,t$, and let $\tilde u_t=-\tilde v_t^2$, where $\tilde v_t$ is the Minkowski linear combination of $v_0=-\sqrt{-u_0}$ and
$v_1=-\sqrt{-u_1}$ as defined by \eqref{defut}.

Then
\begin{equation}\label{eqBMRearr}
u_t\leq\tilde u_t\text{ in }\Omega_t\,.
\end{equation}
\end{thm}

\begin{proof}
We know from Theorem \ref{ThmConv1}
that $v_0$ and $v_1$ are convex functions and
$$
Dv_0(\Omega_0)=Dv_1(\Omega_1)=\R^n\,,
$$
since  $|Du_i|>0$ on $\partial\Omega_i$ for $i=0,1$, by Hopf's lemma.

Let $\tilde v_t$ be the Minkowski linear combination of $v_0$ and $v_1$, defined by \eqref{defut} and
let $\phi$ be a function touching $\tilde v_t$ by above at some point $\bar x\in\Omega_t$. Then by Lemma \ref{lemmaconvolu}
there exists $x_0\in\Omega_0$ and $x_1\in\Omega_1$ such that
$$
\begin{array}{c}
\bar x=(1-t)\,x_0+t\, x_1\,,\\
\tilde v_t(x)=(1-t)\,v_0(x_0)+t\,v_1(x_1)\,,\\
D\tilde v_t(x)=Dv_0(x_0)=Dv_1(x_1)\,.
\end{array}
$$
Up to an approximation argument similar to before, we can assume $D^2u_0(x_0)>0$ and $D^2u_1(x_1)>0$, whence, by \eqref{hessianeconvolu},
$$
D^2\tilde v_t(\bar x)=\left[(1-t)D^2v_0(x_0)^{-1}+t\,D^2v_1(x_1)^{-1}\right]^{-1}\,.
$$
From now on, we can follow the same steps of the proof of Theorem \ref{ThmConvL} with the modifications specified
for Theorem \ref{ThmConv1}, to get that
$\tilde v_t$ satisfies
$$
\left\{\begin{array}{ll}
\tilde v_t^2S_2(D^2\tilde v_t)+\tilde v_tS_1(P(D\tilde v_t)D^2\tilde v_t)\leq1/4\quad&\text{in }\Omega_t\\
\tilde v_t=0\quad&\text{on }\partial\Omega_t\,,
\end{array}
\right.
$$
which implies that $\tilde u_t=-\tilde v_t^2$ is a supersolution of problem \eqref{eq1} when $\Omega=\Omega_t$. Then
\eqref{eqBMRearr} follows from the Comparison Principle.
\end{proof}

Arguing as in the proof of the equality case in the Brunn-Minkowski inequality for the eigenvalue $\Lambda_2$
in \cite{LiuMaXu}, we can prove that inequality \eqref{eqBMRearr} is in general strict in $\Omega_t$ and the functions
 $u_t$ and $\tilde u_t$ coincide in $\Omega_t$ if and only if $\Omega_0$ and $\Omega_1$ are homothetic.

\section{Proof of Theorem \ref{ThmBM1}}
\begin{proof}
First of all, we notice that the quotient in the definition \eqref{deftauk}
admits a minimizer. Indeed, consider the functional
\begin{equation}
{\bf F}(w)=\frac{1}{k+1}\int_{\Omega}(-w)S_k(D^2w)dx-
\int_{\Omega}w\,dx\,;
\label{RigTor2}
\end{equation}
following \cite{W} and \cite{W2}, we know that ${\bf F}$ has a
minimizer $u\in \Gamma_{k,0}(\Omega)$ which solves \eqref{eq1} and also minimizes the quotient
in \eqref{deftauk}. Then
\begin{equation}
\label{aggiunta2}
\tau_k(\Omega)=
\frac{\left[-\int_{\Omega}u\,dx \right]^{k+1}}{\int_{\Omega}(-u)S_k(D^2u)dx}\,.
\end{equation}
Integrating in $\Omega$ the equation in \eqref{eq1} we obtain
\begin{equation}
\label{perparti}
-\int_\Omega u S_k(D^2u)dx=-\int_\Omega u\, dx\,,
\end{equation}
and using (\ref{aggiunta2}) we find the following relation
\begin{equation}
\tau(K)=\left[-\int_\Omega u(x)\,dx
\right]^{k}\,.
\label{forrigtor}
\end{equation}

Next, we notice that $\tau_k:\R^n\to\R_+$ is a positively homogeneous operator of degree $(n+2)k$. Indeed, if $u$ solves problem \eqref{eq1} in $\Omega$,
it is easily seen that the function $$v(x)=\lambda^2u(x/\lambda)$$ solves the same problem in $\lambda\Omega$;
the homogeneity of $\tau_k$ easily follows from \eqref{forrigtor}.
In particular,
$\tau_2:\R^3\to\R_+$ is homogeneous of degree $10$, i.e.
$$\tau_2(\lambda\Omega)=\lambda^{10}\tau_2(\Omega)$$
for every $\lambda>0$ and every convex set $\Omega$.
\bigskip

Now let us denote by $u_i$ the solution of (\ref{eq1}) when $\Omega=\Omega_i$, for
$i=0,1,t$.
Thanks to Theorem \ref{ThmBMRearr} we have
\begin{equation}
\sqrt{-u_t((1-t)x+ty)} \geq
(1-t)\sqrt{-u_0(x)}+t\sqrt{-u_1(y)}
\,,\quad\forall x\in\Omega_0,\,y\in \Omega_1\,.
\label{BMdimostrare}
\end{equation}
Then by the arithmetic-geometric mean inequality it
follows
\begin{equation}
|u_t((1-t)x+ty)|\geq |u_0(x)|^{1-t}|u_1(y)|^t\,,\quad\forall x\in\Omega_0,\,y\in \Omega_1\,.
\label{bmfunzionipr}
\end{equation}
Now, extend $u_i$ as zero in $\R^3\setminus\Omega_i$, for
$i=0,1,t$. Inequality (\ref{bmfunzionipr}) continues to hold; indeed, if
either $x\notin\Omega_0$ or $y\notin\Omega_1$, then the right hand-side vanishes
and the left hand-side is nonnegative. Hence we may apply
the Pr\'ekopa-Leindler inequality and, taking in account \eqref{forrigtor}, we obtain
$$
\tau_2 (\Omega_t) \geq \tau_2(\Omega_0)^{1-t} \tau(\Omega_1)^t\,,
$$
where $\Omega_t=(1-t)\Omega_0+t\Omega_1$. The latter implies \eqref{BM1} thanks to the homogeneity of $\tau_2$, as observed in Section \ref{BMsec}.
\bigskip

Finally, we consider the equality case. Let $\Omega_0$, $\Omega_1$ and $t\in(0,1)$ be such that
equality holds in \eqref{BM1}; then $|u_t|$, $|u_0|$ and $|u_1|$, extended as zero in
$\R^3\setminus\Omega_t$, $\R^3\setminus\Omega_0$ and
$\R^3\setminus\Omega_1$ respectively, give equality in the
Pr\'ekopa-Leindler inequality. Hence,
by Proposition \ref{Prekopaleindler} we deduce that
\begin{equation}
u_1(y)=C\,u_0(ay+y_0)\,,
\label{Ugualert}
\end{equation}
where $C,a>0$ and $y_0\in\R^3$. Since $u_i(x)<0$ if and only if
$x\in \Omega_i$, $i=0,1$, we deduce that $\Omega_0$ and $\Omega_1$ must be
homothetic.

\end{proof}

\section{Proof of Theorem \ref{ThmUrysohn}, that is Minkowski addition as a rearrangment}

Before giving the proof of Theorem \ref{ThmUrysohn}, let us
recall some notions from convex geometry.

To every bounded convex set $K\subset\R^n$ it is associated its {\em support function} $h(K,\cdot) : \R^n\to [0,+\infty)$ in the following way:
$$
h_K(v)= \sup_{x\in K} \langle v,x \rangle, \qquad v\in\R^n.
$$
The support function of a convex set is obviously homogeneous of degree one and, as supremum of linear function, it is convex.
Moreover, for every $a\geq 0$ and every $K,L$ convex sets, it holds
\begin{equation}\label{nuova}
\begin{array}{ll}
& h_{aK} = ah_K,\\
& h_{K+L} = h_K + h_L.
\end{array}
\end{equation}
We refer to \cite{sc} for more details and properties of convex sets and support functions.

The mean width $\mw (\Omega) $ of a convex set $\Omega$ is defined as
$$
\mw (\Omega) = \dfrac{2}{n\wn} \int_{\sfe} h_\Omega(\theta)\, d\Hsfe(\theta),
$$
where $\wn$ is the measure of the unit ball $B_1=\{x\in\R^n\,:\,\|x\|\leq 1\}$ of $\R^n$ and $S^{n-1}$ is the unit sphere of $\R^n$, i.e.
$\sfe=\partial B_1$, while $\Hsfe$ denotes the $(n-1)$-dimensional Hausdorff measure..
We recall that the following Urysohn's inequality holds in the class of convex sets:
\begin{equation}\label{isobV}
\frac{V(\Omega)}{\omega_n}\leq\left(\frac{\mw(\Omega)}{2}\right)^n\,,
\end{equation}
where equality holds if and only if $\Omega$ is a ball.

Given any bounded convex set $\Omega$, we set
$$
\Omega^\sharp=\{x\in\R^n\,:\,\|x\|\leq\mw(\Omega)/2\}\,.
$$

\begin{proof}[Proof of Theorem \ref{ThmUrysohn}.]
Let $\Omega$ be a subset of $\R^3$ with mean width $\mw$ and {\em Steiner point} $\s$.
We recall that the Steiner point $\s(\Omega)$ of a convex set $\Omega$ can be defined as
$$
\s(\Omega) = \dfrac 3{4\pi} \int_{\sfedue} \theta\, h_\Omega(\theta)\, d\Hdue(\theta).
$$
Without loss of generality, we can assume that $\s$ coincides with the origin.

By Hadwiger's Theorem (see \cite{sc}, Section 3.3) there exists a sequence of rotations $\{\rho_N\}$ such that
$$
\Omega_N = \frac 1{N+1} (\rho_0\Omega+...+\rho_N\Omega)
$$
converges to $\Omega^\sharp$ in the Hausdorff metric.

Denote by $u_N$ the solution of problem \eqref{eq1} in $\Omega_N$.

Moreover, set as before $v=-\sqrt{-u}$ and, for every $N\in\N$, let $\tilde v_N$ be the Minkowski combination of the functions
$$
v_0(x)=v(\rho_0^{-1}x),\dots,\,v_N(x)=v(\rho_N^{-1}x)
$$
with ratio
$$t=(1/(N+1),\dots,1/(N+1))\in\Upsilon_N\,.
$$
Then set
$$
\tilde u_N=\tilde v_N^2\,.
$$
By Theorem \ref{ThmBMRearr}, $\tilde u_N$ is a supersolution of the problem solved by $u_N$ and it holds
$$
|u_N|\geq|\tilde u_N|\quad\text{ in }\Omega_N\,.
$$
Notice that the functions $\tilde u_N$ are uniformly bounded and uniformly lipschitz, since
\begin{equation}\label{maxestimate}
\max_{\overline\Omega_n}|\tilde u_N|=\max_{\overline\Omega}|u|\quad\text{ and }
\quad\max_{\overline\Omega_n}|D\tilde u_N|=\max_{\overline\Omega}|Du|\,,
\end{equation}
by the definition of $\tilde v_N$ and easy properties of infimal convolution.

Then, possibly up to a subsequence, they converge uniformly to function $\tilde u$ which is
a supersolution of problem \eqref{eq1} in $\Omega^\sharp$, thanks to the stability of viscosity solution under uniform
convergence. Hence
$$
|u^\sharp|\geq|\tilde u|\quad\text{ in }\Omega^\sharp\,,
$$
whence
\begin{equation}\label{usharplpversusutildelp}
\|u^\sharp\|_{L^p(\Omega^\sharp)}\geq\|\tilde u\|_{L^p(\Omega^\sharp)}\quad\text{for every }p\in(0,+\infty]\,.
\end{equation}
On the other hand, by the definition of $\tilde v_N$ and $\tilde u_N$, it holds
$$
\sqrt{-\tilde u_N\left(\frac1{N+1}\sum_{i=0}^Nx_i\right)}\geq\frac1{N+1}\sum_{i=0}^N\sqrt{-u(\rho_i^{-1}x_i)}\,,
$$
for every $x_i\in\rho_i\Omega$, $i=0,\dots,N$.
This yields
$$
|\tilde u_N(x)|\geq \prod_{i=0}^N|u(\rho_i^{-1}x_i)|^{\frac{1}{N+1}}
$$
for every $x_0,\dots,x_N\in\R^3$ such that $x=\frac{1}{N+1}\sum_{i=0}^Nx_i$,
once we extend $u_N$ and $u$ as zero outside of $\Omega_N$ and $\Omega$, respectively.
Then the Pr\'{e}kopa-Leindler inequality, Proposition \ref{Prekopaleindler}, implies
\begin{equation}\label{prekopa}
\|\tilde u_N\|^p_{L^p(\Omega_N)}\geq\prod_{i=0}^N\left(\int_{\rho_i\Omega}|u(\rho_i^{-1}\xi)|^p\,d\xi\right)^{\frac{1}{N+1}}=
\|u\|^p_{L^p(\Omega)}\quad\text{for every }p\in(0,+\infty]\,.
\end{equation}
Passing to the limit as $N\to\infty$, this yields
$$
\|\tilde u\|_{L^p(\Omega^\sharp)}\geq\|u\|_{L^p(\Omega)}\,,
$$
which jointly with \eqref{usharplpversusutildelp} gives the desired inequality.

Regarding equality conditions, if equality happens in \eqref{ineqlpnorm} for some $p<+\infty$, then
equality must hold in the Pr\'{e}kopa-Leindler inequality \eqref{prekopa} yielding that all the sets
$\rho_i\Omega$ are homothetic. Finally this implies $\rho_i\Omega=\Omega$ for every $i\in\N$, whence
$\Omega_N=\Omega$ for every $N\in\N$ and then $\Omega=\lim_{N\to\infty}\Omega_N=\Omega^\sharp$.

\end{proof}

We could also prove that
$$
\max_{\overline\Omega}|u|=\max_{\overline\Omega^\sharp}|u^\sharp|
$$
implies $\Omega$ is a ball. Indeed, the above equality in $L^\infty$ norms, owing to \eqref{maxestimate}, implies
$$\max_{\overline\Omega}|u|=\max_{\overline\Omega_N}|\tilde u_N|=\max_{\overline\Omega^\sharp}|\tilde u|=\max_{\overline\Omega^\sharp}|u^\sharp|\,.
$$
Then $\tilde u\geq u^\sharp$, $\tilde u$ is a supersolution and $u^\sharp$ is a solution of
\eqref{eq1} in $\Omega^\sharp$, while there exist an interior point $x$ (namely the minimum point of both functions) where
$\tilde u(x)=u^\sharp(x)$. Then, by the strong maximum principle, $\tilde u\equiv u^\sharp$, that is $\tilde u$ is a solution.
This would lead us to the conclusion in a sort of standard way (see for instance \cite{Colesanti}, \cite{CoCuSa} and \cite{LiuMaXu}),
since it forces the equality
$$
D^2v_i(x_i)=D^2v_j(x_j)\quad\,i\neq j
$$
for every $x_i\in\rho_i\Omega,\,x_j\in\rho_j\Omega$ such that $Dv_i(x_i)=Dv_j(x_j)$.

On the other hand, to use the classical strong maximum principle we should prove that $\tilde u$ is actually
a classical supersolution, for which we need to know that $D^2u>0$ in $\Omega$: it is possible to obtain this information thanks to the
{\em constant rank theorem} of \cite{LiuMaXu}.
Notice that, at the moment, with the technique presented here, may be we can prove that a solution
is strictly convex (see \cite{SaLectNotes}), but not that $D^2u>0$.

\section{Generalizations, isoperimetric inequalities, open problems and final remarks}

\subsection{Generalizations}

Let $p>0$, $p\neq 2$, $\lambda>0$ and consider problem \eqref{eqp}. Notice that, possibly by multiplying $u$ by $\lambda^{1/(2-p)}$,
we can always assume $\lambda=1$ and reduce to study the following normalized problem
\begin{equation}\label{eqp1}
\left\{\begin{array}{ll}
S_2(D^2u)=(-u)^p\quad&\text{in }\Omega\,,\\
u=0\quad&\text{on }\partial\Omega\,.
\end{array}
\right.
\end{equation}
Owing to \cite{CW1} (see also \cite{CW2, W2}), since $2=k>n/2=3/2$, we know that for every $p\in(0,2)\cup(2,+\infty)$,
if $\Omega$ is a $\mathcal{C}^+_1$ set, there exists an admissible non-trivial solution $u\in C^{3,\alpha}(\Omega)\cap C^{0,1}(\overline\Omega)$,
which minimizes in $\Phi_{2,0}(\Omega)$
the functional
$$
F_{2,p}(w)=\frac{1}{k+1}\int_{\Omega}(-w)S_2(D^2w)dx-\frac{1}{p+1}
\int_{\Omega}(-w)^{p+1}\,dx\,.
$$
A straightforward argument shows then that $u$ also minimize in $\Phi_{2,0}$ the quotient
$$
Q_{2,p}(w)=\frac{\int_{\Omega}(-w)S_2(D^2w)dx}{(\int_{\Omega}(-w)^{p+1}\,dx)^{3/{p+1}}}\,.
$$

\begin{proof}[Proof of Theorem \ref{ThmConvp}.]
Let
$$v=-(-u)^{\frac{2-p}{4}}\,.
$$
A straightforward calculation shows that $v$ solves
\begin{equation}\label{eqvp}
\left\{\begin{array}{ll}
v^2S_2(D^2v)+\frac{(p+2)}{2-p}\,v\,S_1(P(Dv)D^2v)=\frac{(p-2)^2}{16}&\,\text{ in }\Omega\,,\\
v=0&\,\text{ on }\partial\Omega\,.
\end{array}
\right.
\end{equation}

Now we can consider the convex envelope $\tilde v$ of $v$ and argue exactly as in Theorem \ref{ThmConv1} to get that
$\tilde v$ is an admissible solution of
$$
\left\{\begin{array}{ll}
\tilde v^2S_2(D^2\tilde v)+\frac{(p+2)}{2-p}\,\tilde v\,S_1(P(D\tilde v)D^2\tilde v)\leq\frac{(p-2)^2}{16}&\,\text{ in }\Omega\,,\\
\tilde v=0&\,\text{ on }\partial\Omega\,,
\end{array}
\right.
$$
whence $\tilde u=-(-\tilde v)^{\frac{4}{2-p}}$ is $2$-convex and satisfies
$$
S_2(\tilde u)\leq (-\tilde u)^p\,.
$$
Multiplying by $-\tilde u$ and integrating over $\Omega$, we see that
$$
\frac{\int_{\Omega}(-\tilde u)S_2(D^2\tilde u)dx}{\int_{\Omega}(-\tilde u)^{p+1}\,dx}\leq 1\,,
$$
whence
$$
Q_{2,p}(\tilde u)\leq\|\tilde u\|_{L^{p+1}(\Omega)}^{p-2}\leq \|u\|_{L^{p+1}(\Omega)}^{p-2}\,,
$$
where the second inequality is obvious since $|\tilde u|\geq|u|$ and $p<2$.
Hence $\tilde u$ is a minimizer of $Q_{2,p}$ and then solves \eqref{eqp}.

\end{proof}

Analogously to the definition of $\tau_2$, we set
$$
\tau_ {2,p}(\Omega)=\sup\{Q_{2,p}(w)^{-1}\,:\,w\in\Phi_{2,0}(\Omega),\,w<0\}
$$
and we get
\begin{equation}\label{tau2p}
\tau_{2,p}(\Omega)=Q_{2,p}(u)^{-1}=\|u\|_{L^{p+1}(\Omega)}^{2-p}\,.
\end{equation}
Notice that $\tau_{2,p}$ is positively homogeneous of degree $\frac{p+10}{p+1}$, i.e.
$$
\tau_{2,p}(\lambda\Omega)=\lambda^{\frac{p+10}{p+1}}\tau_{2,p}(\Omega)\quad\text{for every }\lambda\geq 0\,.
$$
We can prove the following Brunn-Minkowski inequality for $\tau_{2,p}$.
\begin{thm}\label{ThmBMp}
Let $\Omega_0$ and $\Omega_1$ be $\mathcal{C}_1^+$ subsets of $\R^3$ and $t\in[0,1]$ and let $p\in(0,2)$.
Then
\begin{equation}\label{BMp}
\tau_{2,p}((1-t)\Omega_0+t\Omega_1)^{\frac{p+1}{p+10}}\geq (1-t)\tau_{2,p}(\Omega_0)^{\frac{p+1}{p+10}}+
t\tau_{2,p}(\Omega_1)^{\frac{p+1}{p+10}}\,.
\end{equation}
Moreover, equality holds in \eqref{BMp} if and only if $\Omega_0$ and $\Omega_1$ are homothetic.
\end{thm}
\begin{proof}
Denote by $u_i$ the solution of (\ref{eq1}) when $\Omega=\Omega_i$ and let $v_i=-(-u_i)^{(2-p)/4}$, for
$i=0,1$. The function $v_i$ solves \eqref{eqvp} in $\Omega_i$, $i=0,1$.
Let $\tilde v_t$ be the Minkowski linear combination of $v_0$ and $v_1$, with ratio $t$; then
$\tilde v_t$ turns out to be a supersolution of \eqref{eqvp} in $\Omega_t$ (the argument is the same as in the previous theorems and
 I leave the details to the reader).
Arguing now as in Theorem \ref{ThmConvp} and using
Pr\'ekopa-Leindler  as in the proof of Theorem \ref{ThmBM1}, we finally get inequality \eqref{BMp}, including equality conditions.
\end{proof}

\subsection{Isoperimetric inequalities}

As observed in \cite[Remark 6.1]{BiaSal}, every Brunn-Minkowski inequality, owing to Hadwiger's Theorem (arguing as in \cite[Corollary 2.2]{BiaSal}),
implies an Urysohn's inequality
which states an optimality property of the ball for the involved inequality among convex sets with given mean width.
These inequalities are {\em sharp}, in the sense that, continuing to
argue similarly to \cite{BiaSal}, we can see that equality holds in them if and only if $\Omega$ is a ball, but in many cases they are not
optimal, in the sense that some other stronger inequality may be proved.
In this context, we have \eqref{UryL} and \eqref{Ury1}, that can be rephrased by saying that,
in the class of regular convex sets with given mean width, $\Lambda_2$ and $\tau_2$ attain respectively the minimum and the maximum
when $\Omega$ is the ball. On the other hand, using a rearrangement technique by quermassintegrals introduced by Tso in \cite{Tso} and further developed
by Trudinger in \cite{Trudi}, one could prove \eqref{IsoL} and \eqref{Iso1}.
As already said in the introduction, the latter inequalities imply \eqref{UryL} and \eqref{Ury1}, due to
classical Urysohn's inequality and to the monotonicity of the involved functionals.
However, it must be noticed that \cite{Tso} is not sufficient to prove \eqref{IsoL} and \eqref{Iso1}, since it treats only convex
functions (and the solutions of \eqref{eqL} and \eqref{eq1} are in general not convex), while the results of
\cite{Trudi} mainly relies on the generalization to $k$-convex sets of isoperimetric inequalities for quermassintegrals, claimed in \cite{Trudi2}, but
whose proof there contained is incomplete. On the other hand, it is possible to apply the rearrangement procedure of
Tso and Trudinger, once proved the convexity of the level sets of the solutions of \eqref{eqL} and \eqref{eq1}.
Then a complete proof of \eqref{IsoL} and \eqref{Iso1}
can be finally obtained by applying the results of \cite{Trudi, Tso}, after
Theorem \ref{ThmConvL} and Theorem \ref{ThmConv1}, giving the following theorem.
\begin{thm}\label{ThmIsop}
Among $\mathcal{C}_1^+$ sets $\Omega$ in $\R^3$ with given surface, the functionals $\Lambda_2$ and $\tau_2$ achieve respectively its minimum
and its maximum value when $\Omega$ is a ball, i.e.
inequalities \eqref{IsoL} and \eqref{Iso1} hold.
\end{thm}
\begin{proof}
To prove \eqref{Iso1}, we define the $1$-symmetrand $u^*_1$ of the solution $u$ in $\Omega$ of
\eqref{eq1} ($k=2$, $n=3$) as in \cite{Trudi} and apply Theorem 3.1 of \cite{Trudi} to obtain that
$u_1^*$ is a supersolution of problem \eqref{eq1} in $\Omega^*$. Then we have just to argue as in Theorem \ref{ThmBM1}.
Notice that Theorem 3.1 of \cite{Trudi} works fine in this case because we know, by Theorem \ref{ThmConv1}, that all the level
sets of $u$ are convex, then in its proof we can use the classical Alexandrov-Fenchel inequalities
(see \cite{sc}) in place of the inequalities from \cite{Trudi2}.

To prove \eqref{IsoL}, we notice that, thanks to Theorem \ref{ThmConvL}, we can apply Theorem 4.1 of \cite{Trudi} and
classical
Alexandrov-Fenchel inequalities to obtain
$$
\frac{-\int_{\Omega^*} u_1^*\,S_2(D^2u_1^*)\,dx}{\int_{\Omega^*}|u_1^*|^{3}dx}\leq\Lambda_2(\Omega)\,,
$$
whence the desired inequality immediately descends.

\end{proof}

\begin{rmk}
Notice that a proof of isoperimetric inequalities for quermassintegrals of {\em starshaped} $k$-convex sets has been recently obtained by
P. Guan and J. Li, see \cite{GuanLi}.
This suggest that it would be interesting to investigate the starshapedness of level sets of solutions of
\eqref{eqL} and \eqref{eq1} when $\Omega$ is a $k$-convex starshaped set.
\end{rmk}
\begin{rmk}
Similar observations to the ones that leaded to Theorem \ref{ThmIsop} can be carried on about the possible comparison between Theorem
\ref{ThmUrysohn} and \cite[Theorem 3.1]{Trudi}; that is: owing to the square root convexity of the solutions to \eqref{eq1},
it is possible to compare the solution in a convex domain $\Omega$ with the solution in a ball $\Omega^*$ with the same
surface area as $\Omega$. Again it would be also interesting to study the starshaped case, after \cite{GuanLi}.
\end{rmk}

\subsection{Open problems}

Finally I would like to mention some interesting open problems that are naturally suggested by the results
here contained.
\medskip

The first obvious question is whether these results can be extended to the general case $2<k<n$ and $n>3$ or not.
I strongly believe that log-convexity of the solutions of \eqref{eqL} and the square-root convexity of the solutions of
\eqref{eq1} hold also for $k>2$ when $n>3$. The easiest case to be treated as a next step should be $k=n-1$ for general dimension
$n$. Another natural question is whether the convexity properties stated in Theorems \ref{ThmConvL} and \ref{ThmConv1} are optimal
or not. In \cite{LiuMaXu} the authors prove that the square-root convexity of the solution is optimal for problem \eqref{eq1}.
\medskip

An interesting question, in connection with any Brunn-Minkowski inequality, is about a possible
Rogers-Shephard inequality for the involved functional. For instance, \eqref{BM1} applied in the case $\Omega_0=\Omega$
and $\Omega_1=-\Omega$ reads
$$
\tau_2(D\Omega)\geq2^{10}\tau_2(\Omega)\,,
$$
where $D\Omega=\Omega+(-\Omega)$ is the so called {\em difference body} of $\Omega$. In the case of volume, the classical
Brunn-Minkowski inequality in $\R^n$ yields $V(D\Omega)\geq 2^nV(\Omega)$. Rogers and Shephard \cite{RoSh} proved
a reverse inequality:
$$
V(D\Omega)\leq\binom{2n}{2}V(\Omega)\,.
$$
Then the question is: does there exist a constant $C$ such that
$$
\tau_2(D\Omega)\leq C\,\tau_2(\Omega)\quad\text{for every convex set }\Omega\subset\R^3\,?
$$
An analogous question rises for $\Lambda_2$ and for any other functionals satisfying a Brunn-Minkowski inequality.
\medskip

One of the most important result in convex geometry regards the solution of the Minkowski problem; roughly speaking,
the Minkowski problem asks to find a convex body with given Gauss curvature in function
of the normal direction to the boundary. Brunn-Minkowski inequality suggests a way to solve the Minkowski problem
with a variational argument; moreover the characterization of equality conditions in the Brunn-Minkowski inequality
yields uniqueness in Minkowski problem.
As observed by Jerison, once proved a Brunn-Minkowski inequality for some functional, it is natural
to pose the question of a related Minkowski problem, see \cite[Section 4]{Colesanti} for a nice presentation of this argument
and more references.
Suitable Minkowski problems are solved for
the Newton capacity \cite{J1}, for the first eigenvalue of the Laplacian \cite{J2} and for torsional rigidity \cite{CF}; a study for the case of $p$-capacity ($1<p<2$) is actually in progress, see \cite{CLSXYZ}.

After Theorem \ref{ThmBM1}, it is then natural to raise the same question for $\tau_2$.
The first basic step in this direction is given by the following representation formula, holding for a $C^2_+$ domain in $\R^n$ and
any $k\in\{1,\dots,n\}$:
\begin{equation}\label{representation1}
\tau_k(\Omega)=\frac{1}{k(n+2)}\int_\sfe h_\Omega(X)\,|Du(\nu_\Omega^{-1}(X))|^{k+1}\,d\sigma^\Omega_{n-k}(X)\,,
\end{equation}
where $u$ is the solution of \eqref{eq1}, $h_\Omega$ is the support function of $\Omega$, $\nu_\Omega$ is the Gauss map
of $\partial\Omega$ (then $\nu^{-1}(X)$ is the point
on $\partial\Omega$ where the outer normal direction is $X$) and
$\sigma_{n-k}^\Omega$ denotes the $(n-k)$-area measure of $\partial\Omega$, whose density
is $S_{n-k}(r_1,\dots,r_{n-1})(X)$ where $r_1,\dots,r_{n-1}$ are the principal radii of curvature of $\partial\Omega$ at the point
$\nu_\Omega^{-1}(X)$.
Formula \eqref{representation1} is just a particular case of \cite[Proposition 4.1]{BNST}.

From \cite[Proposition 4.1]{BNST} it is also easy to obtain an analogous representation formula for $\Lambda_k$:
\begin{equation}\label{representationL}
\Lambda_k(\Omega)=\frac{1}{2k}\int_\sfe h_\Omega(X)\,|Du(\nu_\Omega^{-1}(X))|^{k+1}\,d\sigma^\Omega_{n-k}(X)\,,
\end{equation}
where $u$ is the solution of \eqref{eqL} such that $\int_\Omega|u|^{k+1}dx=1$ and the rest of notation is as before.
As \eqref{representation1} for $\tau_2$, \eqref{representationL} is the first step towards a Minkowski problem for $\Lambda_2$.


\begin{thebibliography}{99}

\bibitem{ALL} G. Alvarez, J.-M. Lasry and P.-L. Lions,
{\it Convex viscosity solutions and state constraints}, J. Math.
Pures Appl. 76 (1997), 265-288.


\bibitem{BianGuan} B. Bian, P. Guan,
{\em A microscopic convexity principle for nonlinear partial differential equations},
Invent. Math. \textbf{177} (2009), 307-335.

\bibitem{BianGuan2} B. Bian, P. Guan,
{\em A structural condition for microscopic convexity principle},
Discrete Contin. Dyn. Syst. \textbf{28} (2010), 789-807

\bibitem{BLS} C. Bianchini, M. Longinetti, P. Salani,
 {\em Quasiconcave solutions to elliptic problems in convex rings}, Indiana Univ. Math. J. \textbf{58} (2009), 1565-1589.

\bibitem{BiaSal} C. Bianchini, P. Salani, {\em Concavity properties for elliptic free boundary problems},
 Nonlinear Anal. \textbf{71} (2009), 4461-4470,

\bibitem{Borell1} C. Borell, {\it Capacitary inequalities of
Brunn-Minkowski type}, Math. Ann. \textbf{263} (1984), 179-184.

\bibitem{b2} C. Borell, {\it Hitting probability of killed
Brownian motion: a study on geometric regularity}, Ann. Sci. Ecole
Norm. Super. Paris \textbf{17} (1984), 451-467.

\bibitem{b3} C. Borell, {\it Greenian potentials and
concavity}, Math. Ann. \textbf{272} (1985), 155-160.

\bibitem{BNST} B. Brandolini, C. Nitsch, P. Salani, C. Trombetti,
{\em Serrin type overdetermined problems: an alternative proof},
Arch. Rat. Mech. Anal. \textbf{190} (2008), 267-280.

\bibitem{bl} H. J. Brascamp, E. H. Lieb, {\it On extensions
of the Brunn-Minkowski and Pr\'{e}kopa-Leindler theorems,
including inequalities for log-concave functions, and with an
application to the diffusion equation}, J. Funct. Anal. \textbf{22} (1976),
366-389.

\bibitem{CGM} L. Caffarelli, P. Guan, X.-N. Ma,
{\em A constant rank theorem for solutions of fully nonlinear elliptic equations},
Comm. Pure Appl. Math. \textbf{60} (2007), 1769-1791

\bibitem{CaffF} L.A. Caffarelli, A. Friedman,
{\em Convexity of solutions of semilinear elliptic equations},
Duke Math. J. \textbf{52} (1985), 431-456.

\bibitem{cjl} L. A. Caffarelli, D. Jerison, E. H. Lieb, {\it
On the case of equality in the Brunn-Minkowski inequality for
capacity}, Adv. Math. \textbf{117} (1996), 193-207.

\bibitem{cns} L. Caffarelli, L. Nirenberg, J. Spruck,
{\em The Dirichlet problem for nonlinear second-order elliptic equations. III. Functions of the eigenvalues of the Hessian},
Acta Math. \textbf{155} (1985), 261-301.

\bibitem{CW1} K.S. Chou, X.J. Wang, {\em A Variational theory for Hessian equations},
Comm. Pure Appl. Math. \textbf{54} (2001), 1029-1064.

\bibitem{CW2} K.S. Chou, X.J. Wang, {\em Variational solutions to Hessian equations},
preprint 1996, available at www.maths.anu.edu.au/research.reports/

\bibitem{Colesanti} A. Colesanti, {\it Brunn-Minkowski
inequalities for variational functionals and related problems},
Adv. Math. \textbf{194} (2005), 105-140.


\bibitem{CF} A. Colesanti, M. Fimiani, {\em The Minkowski problem for
the torsional rigidity}, Indiana Univ. Math. J. \textbf{59} (2010), 1013-1039.

\bibitem{CoCuSa} A. Colesanti, P. Cuoghi, P. Salani, {\it Brunn-Minkowski inequalities for two functionals involving
    the $p$-Laplace operator}, Appl. Anal. \textbf{85} (2006), 45-66.
    
\bibitem{CLSXYZ} A.\ Colesanti, E.\ Lutwak, P.\ Salani, J.\ Xiao, 
D.\ Yang, G.\ Zhang: work in progress.



\bibitem{CS1} A. Colesanti and P. Salani, {\it Quasiconvex envelope of a function and convexity of level sets of solutions
to elliptic equations},
Math. Nachr. \textbf{258} (2003), 3-15.

\bibitem{colesala} A. Colesanti and P. Salani, {\it The
Brunn-Minkowski inequality for $p$-capacity of convex bodies}, Math. Ann. \textbf{327} (2003), 459-479.


\bibitem{dubuc} S. Dubuc, {\it Crit\`{e}res di convexit\'{e}
et in\'{e}galit\'{e}s int\'{e}grales}, Ann. Inst. Fourier (Grenoble) \textbf{27} (1977), 135-165.


\bibitem{gardner} R.J. Gardner, {\it The Brunn-Minkowski
inequality}, Bull. Amer. Math. Soc. (N.S.) \textbf{39} (2002),
355-405.

\bibitem{GardHart} R.J. Gardner, D. Hartenstine, {\it Capacities, surface area, and radial sums},
Adv. Math. \textbf{221} (2009), 601-626.

\bibitem{GuanLi} P. Guan, J. Li, {\em The quermassintegral inequalities for k-convex starshaped domains},
Adv. Math. \textbf{221} (2009), 1725-1732

\bibitem{GuanMa} P. Guan, X.-N. Ma,
{\em The Christoffel-Minkowski problem. I. Convexity of solutions of a Hessian equation},
Invent. Math. \textbf{151} (2003),  553-577.

\bibitem{Hart} D. Hartenstine, {\it Brunn-Minkowski-Type Inequalities Related to the Monge-Ampere Equation},
Adv. Nonlinear Stud. \textbf{9} (2009), 277-294.

\bibitem{K} B. Kawohl, {\em Rearrangements and Convexity of Level Sets
in P.D.E.}, Lecture Notes in Mathematics, \textbf{1150}, Springer,
Berlin, 1985.


\bibitem{Kore} N. J. Korevaar, {\it Convexity of level sets for solutions to elliptic ring problems},
Comm. Partial Differential Equations \textbf{15} (1990), 541-556


\bibitem{Korelewis} N. J. Korevaar and J. L. Lewis, {\it Convex
solutions of certain elliptic equations have constant rank
Hessians}, Arch. Rat. Mech. Anal. \textbf{97} (1987), 19-32.

\bibitem{IS} K. Ishige, P. Salani, {\em Parabolic quasi-concavity for solutions to parabolic problems in convex rings},
Math. Nachr. \textbf{283} (2010), 1526-1548.

\bibitem{J1} D. Jerison, {\em A Minkowski problem for electrostatic capacity},
Acta Math. \textbf{176} (1996), 1--47.

\bibitem{J2} D. Jerison, {\em The direct method in the calculus of variations for convex bodies},
Adv. Math. \textbf{122} (1996), 262-279.

\bibitem{LiuMaXu} P. Liu, X.-N. Ma, L. Xu, {\it A Brunn–Minkowski inequality for the Hessian eigenvalue in three-dimensional convex domain},
Adv. Math. \textbf{225} (2010), 1616-1633.

\bibitem{MaXu} X.-N. Ma, L. Xu, {\it The convexity of solution of a class Hessian equation in bounded convex domain in $\R^3$},
J. Funct. Anal. \textbf{255} (2008), 1713-1723.

\bibitem{Polya} G. P\'{o}lya and G. Szeg\"{o}, {\it
Isoperimetric inequalities in mathematical physics}, Princeton
University Press, Princeton, 1951.

\bibitem{Rocka} R. T. Rockafellar, {\it Convex
analysis}, Princeton University Press, Princeton, New Jersey,
1970.

\bibitem{RoSh} C.A. Rogers, G.C. Shephard,
{\em The difference body of a convex body},
Arch. Math. \textbf{8} (1957), 220-233.

\bibitem{Salani} P. Salani, {\it A Brunn-Minkowski inequality
for the Monge-Amp\`{e}re eigenvalue}, Adv.
Math. \textbf{194} (2005), 67-86.

\bibitem{SaLectNotes} P. Salani, {\em Convexity properties of solutions to PDEs}, Lecture Notes of a course given at the
University of Science and Technology of China, Hefei (China), in preparation.

\bibitem{sc} R. Schneider, {\it Convex bodies: the
Brunn-Minkowski theory}, Cambridge University Press, Cambridge
1993.

\bibitem{SWYY} I.M. Singer, B. Wong, S.-T. Yau, S. S.-T. Yau,
{\em An estimate of the gap of the first two eigenvalues in the Schr\"{o}dinger operator},
Ann. Scuola Norm. Sup. Pisa Cl. Sci. (4) \textbf{12} (1985), 319-333.

\bibitem{Trudi2} N.S. Trudinger, {\em Isoperimetric inequalities for quermassintegrals},
Ann. Inst. H. Poincarè Anal. Non Linèaire \textbf{11} (1994), 411-425.

\bibitem{Trudi} N.S. Trudinger, {\em On new isoperimetric inequalities and symmetrization},
J. Reine Angew. Math. \textbf{488} (1997), 203-220.

\bibitem{Tso} K. Tso, {\em On symmetrization and Hessian equation}, J. d'Anal. Math. \textbf{25} (1989), 94-106.

\bibitem{WangXia} G. Wang, C. Xia, {\it A Brunn-Minkowski inequality for a Finsler Laplacian}, Analysis
(to appear)

\bibitem{W}  X.J. Wang,
{\em A class of fully nonlinear elliptic equations and related functionals},
Indiana Univ. Math. J. \textbf{43} (1994), 25-54.

\bibitem{W2} X.J. Wang, {\em The $k$-Hessian equations}, Lectures Notes in Math. Vol. \textbf{1977}, Springer, Dordrecht 2009.

\end{thebibliography}
\end{document}